\documentclass[final,reqno]{amsart}

\usepackage{amsmath,amsthm,amsfonts,amssymb}
\usepackage{color,comment,soul}
\usepackage{tikz}
\usepackage[color]{showkeys}
\definecolor{refkey}{rgb}{0,1,1}
\definecolor{labelkey}{rgb}{1,0,0}

\newcommand{\eq} [1] {\begin{equation}\label{#1}\quad}
\newcommand{\en} {\end{equation}}
\newcommand{\R}{\mathbb{R}}
\newcommand{\C}{\mathbb{C}}
\newcommand{\N}{\mathbb{N}}
\newcommand{\CS}{\mathbb{C}S}

\renewcommand{\S}{\mathcal{S}}
\newcommand{\diag}{\operatorname{diag}}
\newcommand{\tr}{\operatorname{tr}}
\newcommand{\re}{\operatorname{Re}}
\newcommand{\im}{\operatorname{Im}}

\newcommand{\conv}{\operatorname{conv}}

\newcommand{\scal}[1]{\langle#1\rangle}
\newcommand{\norm}[1]{\left\Vert#1\right\Vert}

\newcommand{\mathsym}[1]{{}}
\newcommand{\unicode}[1]{{}}

\newcommand{\on}[1]{\operatorname{#1}}

\theoremstyle{plain}
\newtheorem{theorem}{Theorem}[section]

\newtheorem{corollary}[theorem]{Corollary}
\newtheorem{lemma}[theorem]{Lemma}
\newtheorem{proposition}[theorem]{Proposition}

\theoremstyle{definition}

\newtheorem{example}[theorem]{Example}

\theoremstyle{remark}
\newtheorem{remark}[theorem]{Remark}
 \numberwithin{equation}{section}
\newtheorem*{ack}{Acknowledgment}

\begin{document}
\title[Inverse Continuity on the Boundary of the Numerical Range]{Inverse Continuity on the Boundary of the Numerical Range}
\author[T. Leake, B. Lins and I. M. Spitkovsky]{Timothy Leake, Brian Lins$^*$, Ilya M. Spitkovsky}
\date{}
\address{Timothy Leake, Yale University}
\email{timothy.leake@gmail.com}
\address{Brian Lins, Hampden-Sydney College}
\email{blins@hsc.edu}
\thanks{$^*$Corresponding author.}
\address{Ilya Spitkovsky, College of William \& Mary and New York University, Abu Dhabi}
\email{ilya@math.wm.edu, imspitkovsky@gmail.com, ims2@nyu.edu}
\subjclass[2000]{Primary 15A60, 47A12; Secondary 54C08}
\keywords{Field of values; numerical
range; inverse continuity; weak continuity}
\thanks{This work was partially supported by NSF grant DMS-0751964.}
\thanks{The third author was supported in part by the Plumeri Award for Faculty Excellence from the College of William and Mary and by Faculty Research funding from the Division of Science and Mathematics, New York University Abu Dhabi.}

\begin{abstract}
Let $A \in M_n(\C)$.  We consider the mapping $f_A(x)=x^*Ax$, defined on the unit sphere in $\C^n$.  The map has a multi-valued inverse $f_A^{-1}$, and the continuity properties of $f_A^{-1}$ are considered in terms of the structure of the set of pre-images for points in the numerical range. It is shown that there may be only finitely many failures of continuity of $f_A^{-1}$, and conditions for where these failure occur are given. Additionally, we give a necessary and sufficient condition for weak inverse continuity to hold for $n=4$ and a sufficient condition for $n>4$.
\end{abstract}

\maketitle

\section{Introduction}
Let $M_n(\C)$ stand for the algebra of $n\text{-by-} n$ matrices with the entries in the complex field $\C$. For any $A\in M_n(\C)$, the quadratic form $f_A$ is defined on the standard $n$-dimensional vector space $\C^n$ by \eq{fA} f_A(x)= x^*Ax. \en The image of the unit sphere $\C S^n$ of $\C^n$ under $f_A$ is by definition the {\em numerical range} (or the {\em field of values}) of $A$, denoted $F(A)$. The continuity of $f_A$ immediately implies such basic properties of $F(A)$ as its connectedness and compactness. In addition, $F(A)$ is a convex set (the classical Toeplitz-Hausdorff theorem)  and, more specifically, the convex hull of a certain algebraic curve $C_A$ associated with $A$, see \cite{Ki} or its English translation \cite{Ki08}. A complete description of all possible shapes for $F(A)$ is given in \cite{HelSpit}.

\iffalse For our purposes, the following classification of the points $z$ on the boundary $\partial F(A)$ of $F(A)$ is convenient. Consider the intersection $\gamma_\epsilon(z)$ of the $\epsilon$-neighborhood of $z$ with $\partial F(A)$. Then $z$ is a {\em flat point} (resp.,   a {\em corner} point) if, for sufficiently small $\epsilon$, $\gamma_\epsilon(z)$  is a line segment (resp., consists of two line segments meeting at $z$ at an angle different from $\pi$); all other points $z\in\partial F(A)$ are by definition {\em round boundary} points. Observe that for such a point $\gamma_\epsilon(z)$ still may contain a line segment with $z$ as its endpoint; we say that $z$ is a {\em fully round} boundary point if this is not the case. \fi 

In this paper, we are concerned with the continuity properties of the inverse (multivalued) function \[ f_A^{-1}\colon F(A)\longrightarrow \CS^n. \] Following \cite{CJKLS2}, we distinguish between strong and weak continuity of $f_A^{-1}$ at $z\in F(A)$, the former (resp., latter) meaning that, with respect to the relative topology on $F(A)$, the direct mapping $f_A$ is open at all (resp., some) pre-images $x\in f_A^{-1}(z)$. 

It was shown in \cite[Theorem 4]{CJKLS2} that the strong continuity holds everywhere on $F(A)$ except maybe the round points of its boundary $\partial F(A)$. (We follow the terminology of \cite{CJKLS2} according to which $z\in\partial F(A)$ is a \textit{round point} if for any $\epsilon>0$ at least one of its one-sided neighborhoods in $\partial F(A)$ is not a line segment.  {A maximal subset of the boundary that is a nontrivial closed line segment is a \textit{flat portion}}. In what follows we will also distinguish the {\em fully round} boundary points, for which both one-sided neighborhoods of $z$ are not line segments.) If there are no round points, that is, $F(A)$ is a polygon, the strong continuity therefore holds everywhere on $F(A)$ \cite[Corollary 5]{CJKLS2}. This result covers in particular all normal matrices.

Moreover, it is a simple observation (also made in \cite{CJKLS2}) that strong continuity holds whenever $f_A^{-1}(z)$ has rank one. Since this is the case
for all $z\in \partial F(A)$ when $A\in M_2(\C)$ and not normal or  $A\in M_3(\C)$ and unitarily irreducible, the strong continuity holds everywhere on $F(A)$ for all $2\text{-by-} 2$ and unitarily irreducible $3\text{-by-} 3$ matrices $A$ (\cite{CJKLS2}, Corollary~6 and Theorem~10, respectively). 

However, the strong (and even weak) continuity may indeed fail at round points $z\in \partial F(A)$ with $f_A^{-1}(z)$ not lying in a one-dimensional subspace; we will call such points {\em multiply generated}. By way of examples it was established in \cite{CJKLS2} that the following is possible: strong continuity failure for unitarily irreducible $4\text{-by-} 4$ matrices and $3\text{-by-} 3$ unitarily reducible matrices, weak continuity failure for $6\text{-by-} 6$ unitarily irreducible and $4\text{-by-} 4$ unitarily reducible matrices. While weak continuity indeed persists for all $A\in M_3(\C)$ \cite[Theorem 11]{CJKLS2}, it remained unclear what was the smallest size of unitarily irreducible matrices for which weak continuity may fail. Another open question is the possible number of such points.

Prompted by the last question, the multiply generated boundary points were further studied in \cite{LLS12}. It was proved there, among other things, that the sharp bound for the number of such points for unitarily irreducible $A\in M_n(\C)$ is $n-3$ if $n=3,4,5$, while already for $n=6$ all boundary points may be multiply generated.

Here we take the matter further, and establish that intrinsic reasons for inverse continuity failure lie deeper, and even at multiply generated points the continuity may persist. We give necessary and sufficient conditions for weak or strong continuity to fail. As one of the applications, it is shown that weak continuity may indeed fail for unitarily irreducible $A\in M_4(\C)$.  

\section{Inverse Continuity}

{ We start with some facts concerning the algebraic curve $C_A$ the convex hull of which is $F(A)$.}  This curve coincides with critical values of the map \eqref{fA} when the domain { $\CS^n$ and the target space $\C$} are treated as real manifolds \cite{JoswigStraub98}.   Let $\Sigma_A$ denote the set of all critical values of $f_A$, then the curve $C_A \subseteq \Sigma_A$.  In fact, if we include the bitangent set $C_A'$ containing all line segments connecting two points in $C_A$ when the points have the same tangent, then $\Sigma_A = C_A \cup C_A'$ \cite{JoswigStraub98}.  { The following} description of the sets $C_A$ and $\Sigma_A$ {is based on} \cite{Ki} and \cite{JoswigStraub98}; see \cite[Section 5]{GaSe12} for more {details.} 

For any $A \in M_n(\C)$, let $\re A = \frac{1}{2} (A+A^*)$ and $\im A = \frac{1}{2i} (A-A^*)$.  The critical value set $\Sigma_A$ is $\{f_A(x)\colon x \text{ is a unit eigenvector of } \re(e^{-i\theta} A) \}$.

By a result of Rellich \cite{Rellich69} (see also \cite[Section 3.5.4, Corollary 2]{Baumgartel}), for each matrix $\re(e^{-i\theta}A)$, there is an orthonormal basis of eigenvectors $\{x_k(\theta)\}$, such that the eigenvectors are analytic functions of $\theta \in [0, 2\pi]$.  { The respective eigenvalues $\lambda_k(\theta)$ also depend on $\theta$ analytically; we will call them the {\em eigenfunctions} of $A$.} 

The boundary generating curve $C_A$ is the union of the analytic curves 
\begin{equation} \label{eq:eigcurve}
z_k(\theta) = {f_A(x_k(\theta)) =} e^{i\theta} (\lambda_k(\theta) + i \lambda_k'(\theta)),
\end{equation}
where $\lambda_k'$ is the derivative of $\lambda_k$ with respect to $\theta$.  Note that there is a permutation $\tau$ of $\{1,\ldots, n\}$ such that $\lambda_k(\theta + \pi) = -\lambda_{\tau(k)}(\theta)$.  Thus $z_k(\theta+\pi) = z_{\tau(k)}(\theta)$.  Therefore $C_A = \bigcup_k \{z_k(\theta) : \theta \in [0,\pi)\}$.

We are primarily concerned with eigenfunctions that correspond to points on the boundary.  The relationship between the eigenfunctions $\lambda_k(\theta)$ and the critical curves $z_k(\theta)$ given by \eqref{eq:eigcurve} helps make certain things clear.  For $\theta \in [0,2\pi)$, let $\ell_\theta$ denote the support line of $F(A)$ with the property that $e^{-i\theta} \ell_\theta$ is the vertical support line to the left of $e^{-i \theta} F(A)$.  Note that any $z_k(\theta)$ given by \eqref{eq:eigcurve} lies on $\ell_\theta$ if and only if $\lambda_k(\theta)$ is minimal.   Furthermore, $\ell_\theta$ contains a flat portion of the boundary when there is more than one minimal eigenfunction at $\theta$ and the minimal eigenfunctions split at the first power.   We will say that an eigenfunction $\lambda_k(\theta)$ \textit{corresponds to $z$ at} $\theta$ if $z_k(\theta) = z$.

Since the eigenfunctions $\lambda_k(\theta)$ are analytic, any two of them may coincide at only finitely many values of $\theta$ unless they are identical.  Moreover, for all but finitely many exceptional points, $\re(e^{-i\theta}A)$ must have exactly $m \le n$ distinct eigenvalues. Inverse continuity failures can only occur at exceptional points where two or more of the distinct eigenfunctions $\lambda_k(\theta)$ coincide.  Let $\theta_0$ denote such an exceptional point. As $\theta$ varies from $\theta_0$, the eigenfunctions will ``split".  Since the eigenfunctions are analytic, we may consider their Taylor series expansions about $\theta_0$.  We will say that a collection of eigenfunctions \textit{splits at power $k$} if the first difference in the coefficients of their Taylor series expansions occurs at a term with order $k$. With this terminology, we are ready to state our main theorem.

\begin{theorem} \label{thm:characterization}
Let $A \in M_n(\C)$ and $z \in \partial F(A) \cap \ell_{\theta_0}$.
\begin{enumerate}
\item $f_A^{-1}$ is strongly continuous at $z$ if and only if $z$ is in the relative interior of a flat portion of the boundary or the eigenfunctions corresponding to $z$ at $\theta_0$ do not split.  
\item $f_A^{-1}$ is weakly continuous at $z$ if and only if $z$ is in a flat portion of the boundary or there exists an eigenvalue function $\lambda_k(\theta)$ that is minimal (pointwise) in a neighborhood of $\theta_0$.
\end{enumerate}
\end{theorem}

The following remark shows that in order for weak continuity to fail, it is sufficient for the minimal eigenvalue functions to split at an odd power\footnote{An earlier version of Theorem \ref{thm:characterization} incorrectly asserted that weak continuity fails at a fully round boundary point if and only if the corresponding eigenvalue functions split at odd degree. See Examples \ref{ex:weakfail2} and \ref{ex:irred}. See also \cite{LLS2.5}.}  

\begin{remark} \label{rem:lexico}
  Each eigenvalue function $\lambda_k(\theta)$ has Taylor series expansion 
\begin{equation}\label{te} 
\lambda_k(\theta) = \lambda_k^{(0)}+\lambda_k^{(1)}(\theta-\theta_0)+\lambda_k^{(2)}(\theta-\theta_0)^2+\ldots. 
\end{equation} 
  For values of $\theta$ sufficiently close to $\theta_0$ with $\theta > \theta_0$, the minimal eigenfunction is the one with the minimal sequence of Taylor coefficients $(\lambda_k^{(0)},\lambda_k^{(1)},\lambda_k^{(2)},\ldots)$ in lexicographical ordering.  For values of $\theta$ sufficiently close to $\theta_0$ with $\theta < \theta_0$, the minimal eigenfunction has the minimal sequence of alternating Taylor coefficients $(\lambda_k^{(0)},-\lambda_k^{(1)},\lambda_k^{(2)},-\lambda_k^{(3)}, \ldots)$ in lexicographical ordering. In particular, if the minimal eigenfunction at $\theta_0$ splits at odd degree, then no one eigenfunction is minimal in a neighborhood of $\theta_0$. If the degree of the splitting is one, then there is a flat portion of the boundary corresponding to angle $\theta_0$. If the splitting is at an odd power greater than one, then weak continuity fails at the corresponding point $z \in \partial F(A)$.   
\end{remark}

\color{black}

\begin{remark}
Note that $\re(e^{-i \theta} A) =H\cos\theta  + K\sin\theta$, where $H = \re A$ and $K = \im A$.  To compute the series expansions for the eigenfunctions corresponding to $z \in \partial F(A) \cap \ell_\theta$, it is typically convenient to replace $A$ by $e^{-i \theta} A$.  Then we may assume that $z \in \ell_0$, and the Taylor series expansions for the eigenfunctions are centered at $0$.  Furthermore, since $H\cos\theta + K\sin\theta = (H + K\tan\theta)\cos\theta$, the eigenfunctions in question are all scalar multiples of the eigenfunctions of $H + t K$ where $t = \tan \theta$.   It is easily checked that if a collection of eigenfunctions of $H+tK$ split at power $k$, then the corresponding eigenfunctions of $H\cos\theta + K\sin\theta$ also split at power $k$.
\end{remark}

Since distinct eigenfunctions $\lambda_k(\theta)$ can only coincide at finitely many exceptional values of $\theta$, Theorem \ref{thm:characterization} implies the following.

\begin{corollary}
For $A \in M_n(\C)$, $f_A^{-1}$ is strongly continuous everywhere on $F(A)$, except possibly finitely many points on $\partial F(A)$.
\end{corollary}

A result of von Nuemann and Wigner \cite{vNW29} known as the ``non-crossing rule'' asserts that {the set of matrices $A$ for which a Hermitian linear pencil $\{ \re A + t\im A\colon t\in\R\}$ contains no matrices with repeated eigenvalues is open and dense in $M_n(\C)$}; see  \cite{FrRoSy84} and \cite{JAG98}   for a deeper discussion and further use of this fact.  {From here the following proposition immediately follows.  }

\begin{proposition}
The set of matrices $A \in M_n(\C)$ for which $f_A^{-1}$ is strongly continuous everywhere on $F(A)$ is generic, that is, it contains an open, dense subset of $M_n(\C)$.   
\end{proposition}

Let $A \in M_n(\C)$ and suppose that $z \in \partial F(A)$ is multiply generated.  We will say that $z$ is an \textit{isolated} multiply generated boundary point if there is a neigborhood of $z$ containing no other multiply generated boundary points.  Note that when $A$ is unitarily irreducible and $n < 6$, all multiply generated, fully round, boundary points are isolated \cite{LLS12}.

\begin{theorem} \label{thm:multiGen}
    Let $A \in M_n(\C)$.  Strong continuity of $f_A^{-1}$ fails at all isolated, fully round,  multiply generated boundary points.
\end{theorem}

\begin{proof}
Let $z$ be a multiply generated, fully round, boundary point. If the eigenfunctions of $\re(e^{-i \theta} A)$ corresponding to $z$ split, then by Theorem \ref{thm:characterization} strong continuity fails.  If the eigenfunctions do not split, then they are identical in a neighborhood of $z$, and that implies that all of the boundary points in that neighborhood are multiply generated, contradicting the assumption that $z$ is isolated.
\end{proof}

\section{Proof of Main Result}

In this section we present some geometric lemmas and finish with a proof of Theorem \ref{thm:characterization}.

\begin{lemma} \label{lemma:sphericalCap}
Let $S$ be a 2-sphere with radius $r>0$ in a real normed space $(X,||\cdot||)$ and let $x
\in S$.  For any $\epsilon \in (0, 2]$, let $C_\epsilon = \{y \in S: ||x-y|| \le
\epsilon r\}$.  If $T\colon X \rightarrow \R^2$ is a linear
transformation, then $T(C_\epsilon) = T(\conv C_\epsilon)$. In particular, $T(C_\epsilon)$ is convex.
\end{lemma}
\begin{proof}
Let $\partial C_\epsilon = \{y \in S : ||x-y|| = \epsilon r \}$.  Since
$\partial C_\epsilon$ is a circle, $T(\partial C_\epsilon)$ is
either a line segment or an ellipse.  { In the former case, it is clear that $T(\partial C_\epsilon) = T(\conv \partial C_\epsilon)$}. In the latter case, any points inside the ellipse
$T(\partial C_\epsilon)$ are contained in $T(C_\epsilon)$ because
$T(C_\epsilon)$ must be simply connected.  Thus $T(C_\epsilon) =
T(C_\epsilon \cup \conv \partial C_\epsilon)$.  The set $C_\epsilon
\cup \conv \partial C_\epsilon$ is homeomorphic to $S^2$, so it
separates the three dimensional affine span of $C_\epsilon$ into interior and exterior components.  The union of
$C_\epsilon \cup \conv \partial C_\epsilon$ with its interior
component is precisely $\conv C_\epsilon$.  It follows that $T(\conv
C_\epsilon) = T(C_\epsilon \cup \conv \partial C_\epsilon) =
T(C_\epsilon)$.
\end{proof}

The following lemma is a slight generalization of the famous Toeplitz-Hausdorff theorem and is inspired by the proof of that theorem found in \cite{Dav71}.
\begin{lemma} \label{lemma:convexBall}
For $A \in M_n(\C)$ and $x \in \CS^n$, let $B_r$ be a ball of radius $r > 0$ around $x$ in $\C^n$.  The image $f_A( \CS^n \cap B_r)$ is a convex set.
\end{lemma}
\begin{proof}
For $y \in \C^n$, let $g(y) = yy^*$ { considered as the mapping} $g\colon \C^n \rightarrow \mathcal{S}_n$ { into the set $\S_n$} of $n$-by-$n$ Hermitian matrices with inner product $\langle X,Y \rangle = \tr(XY)$.  Let $\norm{X}_F = \scal{X,X}^{1/2}$ denote the Frobenius norm on $\S_n$.  Note that $\norm{xx^* - yy^*}^2_F = 2 - 2|x^*y|^2$, and therefore there is an $r' > 0$ such that every $y \in \CS^n \cap B_r$ satisfies $\norm{xx^* - yy^*}< r'$.

{ With this notation at hand,} the map { \eqref{fA} can be rewritten as} $f_A(x) = x^*Ax = \tr(Axx^*) = \hat{f}_A \circ g(x)$ where $\hat{f}_A:\S^n \rightarrow \C$ is the real linear map $\hat{f}_A(X) = \tr(AX)$.

Choose any $z_1, z_2 \in f_A(\CS^n \cap B_r)$.  There exist $y_1,y_2 \in \CS^n \cap B_r$ such that $f_A(y_1) = z_1$ and $f_A(y_2) = z_2$.  Let $V = \text{span}\, \{y_1, y_2 \}$. In \cite{Dav71}, it is observed that $g(\CS^n \cap V)$ is a 2-sphere in $\S_n$.  Let $C_\epsilon = g(\CS^n \cap B_r \cap V)$.  Note that $C_\epsilon$ is a spherical cap which contains $y_1y_1^*$ and $y_2y_2^*$.  Lemma \ref{lemma:sphericalCap} implies that $\hat{f}_A(C_\epsilon)$ is a convex set.  Both $z_1 = \hat{f}_A(y_1y_1^*)$ and $z_2 = \hat{f}_A(y_2y_2^*)$ are contained $\hat{f}_A(C_\epsilon)$, therefore the line segment connecting $z_1$ to $z_2$ is as well.  Since $\hat{f}_A(C_\epsilon) = f_A(\CS^n \cap B_r \cap V) \subset f_A(\CS^n \cap B_r)$, it follows that any pair of points in $f_A(\CS^n \cap B_r)$ is connected by a line segment in $f_A(\CS^n \cap B_r)$, and therefore $f_A(\CS^n \cap B_r)$ is convex.
\end{proof}

\begin{proof}[Proof of Theorem \ref{thm:characterization}.]
If $z \in F(A)$ is not a round boundary point, then the proof that $f_A^{-1}$ is strongly continuous at $z$ can be found in \cite[Theorem 4]{CJKLS2}. Therefore, assume that $z$ is a round boundary point and $z = z_k(\theta_0)$ where $z_k(\theta)$ is given by \eqref{eq:eigcurve} for some $k \in \{1,\ldots,n\}$. 
%{\color{gray}The theorem is already proved  for points in the relative interior of flat portions of the boundary, and for points at the intersection of two flat portions \cite[Theorem 4]{CJKLS2}.  The remaining points are the round points of the boundary.  If $z$ is a round point, then at least one of the two portions of the boundary adjacent to $z$ is parametrized by one of the analytic curves $z_k(\theta)$ given by \eqref{eq:eigcurve}. If both portions of the boundary adjacent to $z$ are parametrized by such curves, then  $z$ is a fully round point.  If one of the two portions is a flat portion of the boundary, then $z$ is half-flat. }
{Let $J$ denote the set of indices $j \in \{1,\ldots,n\}$ such that $z = z_j(\theta_0)$. }  We divide our proof into two cases.  \\

\noindent
\textit{Case I. Suppose the eigenfunctions $\{\lambda_j(\theta):j \in J\}$ do not split at $\theta_0$}. Fix $k \in J$.  Any $x \in f_A^{-1}(z)$ must be contained in the $\text{span}\, \{x_j(\theta_0) : j \in J\}$ where $x_j(\theta)$ denotes the analytic family of eigenvectors corresponding to $\lambda_j(\theta)$.  In particular there must be constants $\beta_j$ such that $x = \sum_{j \in J} \beta_j x_j(\theta_0)$ and $\sum_{j \in J} |\beta_j|^2 = 1$.  Then $x(\theta) = \sum_{j \in J} \beta_j x_j(\theta)$ satisfies $z_k(\theta) = f_A(x(\theta))$ for all $\theta$.  If $V$ is any open ball around $x$ in $\CS^n$, then $f_A(V)$ contains a relatively open neighborhood of $z$ along the curve $z_k(\theta)$. {If $z$ is fully round, then $z_k(\theta)$ completely parametrizes the boundary of $F(A)$ in a neighborhood of $z$}. By Lemma \ref{lemma:convexBall}, $f_A(V)$ must contain a neighborhood of $z$ in $F(A)$ which proves that $f_A^{-1}$ is strongly continuous at $z$.  If $z$ is an endpoint of a flat-portion of the boundary, then the curve $z_k(\theta)$ will only parametrize the curved portion of the boundary adjacent to $z$.  On the other side of $z$ will be a flat portion of $\partial F(A)$.  Any such flat portion will be a line segment connecting $z$ to some other $w \in \partial F(A)$.  If $y \in f_A^{-1}(w)$, then both $x$ and $y$ are eigenvectors of $\re(e^{-i\theta_0}A)$ corresponding to $\lambda_k(\theta_0)$. Let
$$y(t) = \frac{t y + (1-t)x}{||t y + (1-t)x||}, \text{ for } t \in [0,1].$$
The function $f_A(y(t))$ parametrizes the flat portion of the boundary between $w$ and $z$.  In particular, we may choose $\epsilon > 0$ small enough so that $y(\epsilon) \in V$.  Then $f_A(V)$ will contain all points on the flat portion of the boundary sufficiently close to $z$. Since $f_A(V)$ is convex and contains both $f_A(y(\epsilon))$ and the curved portion of the boundary adjacent to $z$, it follows that $f_A(V)$ contains a neighborhood of $z$ in $F(A)$, proving the strong continuity of $f_A^{-1}$ at $z$.  \\

\noindent
\textit{Case II. If the eigenfunctions $\{\lambda_j(\theta):j \in J \}$ split at $\theta_0$, then we will now demonstrate that strong continuity must fail}. Choose a neighborhood of $z$ in the boundary of $F(A)$.  Any such neighborhood will be a union of two one-sided neighborhoods on either side of $z$.  If $z$ is not fully round, one of the two neighborhoods will be a line segment in a flat portion of the boundary.  If $z$ is fully round, both of the one-sided neighborhoods adjacent to $z$ are curved portions of the boundary.  Assuming our neighborhood is small enough, any one-sided curved portion of the neighborhood is parametrized by at least one analytic curve $z_k(\theta)$ for $k \in J$.  Let $k$ be such an index, and consider a sequence $\theta_p \rightarrow \theta_0$ and a converging sequence $y_p \in \CS^n$ such that $f_A(y_p) = z_k(\theta_p) \in \partial F(A)$, for all $p \in \N$.  Let $J_k = \{j \in \{1,\ldots, n \} : z_j(\theta) = z_k(\theta) \text{ identically} \}$.  Since the sequence $\theta_p$ is infinite, we may assume that none of the $\theta_p$ are exceptional points.  Therefore $y_p \in \text{span}\, \{x_j(\theta_p) : j \in J_k \}$.  In particular, any accumulation point of $y_p$ is contained in $\text{span}\, \{x_j(\theta_0) : j \in J_k \}$.  Now consider any other $l \in \{1,\ldots,n\} \backslash J_k$ for which $z_l(\theta_0) = z$.  By assumption, such an index exists.  If we take a sufficiently small neighborhood around $x_l(\theta_0)$, that neighborhood will not intersect the subspace $\text{span}\, \{x_j(\theta_p) : j \in J_k \}$.  Therefore strong inverse continuity fails at $z$.  It is clear, however that if $z$ is not fully round, then any open ball $V$ around $x_k(\theta_0)$ will have $f_A(V)$ contain a neighborhood of $z$ by the same argument as in Case I.  In that case weak continuity holds for $f_A^{-1}$ at $z$. The same is true if there is one $z_k(\theta)$ that parametrizes both one sided neighborhoods of $z$ when $z$ is a fully round point.  

On the other hand, if $z$ is fully round and the minimal eigenfunctions $\lambda_k(\theta)$ and $\lambda_l(\theta)$ on either side of $\theta_0$ differ, then the two curved portions of the boundary on either side of $z$ are parametrized by distinct curves $z_k(\theta)$ and $z_l(\theta)$ corresponding to $\lambda_k(\theta)$ and $\lambda_l(\theta)$. In that case, the same argument given to demonstrate the failure of strong continuity at the beginning of Case II also establishes the failure of weak continuity at $f_A^{-1}$ at $z$.
\end{proof}
\color{black}

\section{Weak Continuity Failures}
Now we turn our attention to {further} developing criteria for the failure of weak continuity at one of the (finitely many) points at which it may fail. {Although Theorem \ref{thm:characterization} provides necessary and sufficient theoretical conditions for the failure of weak continuity on the boundary, in this section we offer more easily checked conditions.}  For $A \in M_3(\C)$, $f_A^{-1}$ is weakly continuous everywhere in $F(A)$ \cite[Theorem 10]{CJKLS2}. This is not necessarily the case in higher dimensions. We provide a complete description of weak continuity for unitarily irreducible 4-by-4 matrices and a sufficient condition for weak continuity to hold at any given point in $\partial F(A)$ which is easy to compute for all $A\in M_n(\mathbb{C})$.

As demonstrated by Theorem \ref{thm:characterization}, if $z \in \partial F(A)$ is not a fully round boundary point, then weak continuity holds at $z$.  If $z$ is a fully round boundary point, we may assume by translation and rotation that $z=0$, and the numerical range is contained in the closed right half-plane.  Let $H = \re A$ and $K = \im A$.  By unitary similarity, we may { further} assume that
\begin{equation} \label{necandsufconditions}
H=\begin{pmatrix}
0&0\\
0&H_1\\
\end{pmatrix},\iffalse  \quad K=\begin{pmatrix}
0 & K_0\\
K_0^* & K_1\\
\end{pmatrix}.\fi 
\end{equation}
{ where $H_1\in M_k(\C)$ is positive definite (and thus invertible) for some $k\le n-2$.} 
{ The left upper block in the respective block representation of $K$ is also zero, since  $z=0$ is a fully round boundary point of $F(A)$. In other words, 
\begin{equation} \label{necandsufconditions1}
K=\begin{pmatrix}
0 & K_0\\
K_0^* & K_1\\
\end{pmatrix}.
\end{equation}
}
\begin{theorem}
\label{necandsuf}
Let $A\in M_n(\mathbb{C})$ with $H= \re A$ and $K = \im A$.  Suppose that $H$ and $K$ have block forms given by {\em \eqref{necandsufconditions},  \eqref{necandsufconditions1}}. Then a sufficient condition for weak continuity to hold at zero is that the Hermitian matrix $K_0 H_1^{-1} K_0^*$ has a simple largest eigenvalue. If $A$ is unitarily irreducible and $n=4$, this condition is also necessary.
\end{theorem}

Before proving the theorem, let us review some relevant facts concerning analytic perturbations of Hermitian matrices.  Consider an analytic perturbation $T(x)$ of a matrix $T^{(0)} = T(0) \in M_n(\C)$ such that $T(x)$ is Hermitian when $x \in \R$.  Such a perturbation has a power series representation,
$$
T(x) = T^{(0)} + x T^{(1)} + x^2 T^{(2)} + \ldots .
$$
Let $\lambda$ be an eigenvalue of $T^{(0)}$ with multiplicity $m$.  We know that the eigenvalue $\lambda$ may split into several eigenvalues of $T(x)$ as $x$ moves away from $x=0$.  Let us denote these eigenvalues $\{ \lambda_j(x) \}_{j = 1}^m$.  Following \cite{Kato}, we call this set the \textit{$\lambda$-group}. Since $T(x)$ is Hermitian for real $x$, each eigenvalue is an analytic function of $x$ in a neighborhood of $0$ \cite[VII-\textsection 3.1]{Kato}.  Thus each $\lambda_j(x)$ has a Taylor series of the form
$$\lambda_j(x) = \lambda + \lambda_j^{(1)}x + \lambda_j^{(2)}x^2 + \ldots.$$
A reduction formula for recursively computing the coefficients of the series above can be found in \cite[II-\textsection 2.3]{Kato}.

Let $\Gamma$ be a positively oriented circle enclosing $\lambda$ but no other eigenvalues of $T^{(0)}$. The matrix
$$P(x) = - \frac{1}{2\pi i} \int_\Gamma (T(x) - \xi I)^{-1} \, d \xi$$
is an orthogonal projection equal to the sum of all eigenprojections of $T(x)$ corresponding to eigenvalues of $T(x)$ lying inside $\Gamma$. Thus the range of $P(x)$ is an invariant subspace of $T(x)$.  Following \cite{Kato} we define
$$
\widetilde{T}^{(1)}(x) = \frac{1}{x} (T(x) - \lambda I) P(x).
$$
The reduction comes from noting that the eigenvalues of $T(x)$ in the $\lambda$-group each have the form $\lambda_j(x) = \lambda + x \lambda_j^{(1)}(x)$ where $\lambda_j^{(1)}(x)$ is an eigenvalue of $\widetilde{T}^{(1)}(x)$ with eigenvector in the range of $P(x)$.  Furthermore, $\widetilde{T}^{(1)}(x)$ is itself analytic in a neighborhood of $x=0$, so it has a power series expansion:
$$
\widetilde{T}^{(1)}(x) = \widetilde{T}^{(1)} + \widetilde{T}^{(2)}x + \widetilde{T}^{(3)}x^2 + \ldots .
$$
Formulas for computing each $\widetilde{T}^{(j)}$ can be found in \cite[p. 78]{Kato}.

\begin{proof}[Proof of Theorem \ref{necandsuf}.]
Consider $T(x) = H + xK$.  In the notation above, $T^{(0)} = H$, $T^{(1)} = K$, and $T^{(j)} = 0$ for all $j \geq 2$.  Using the formulas from \cite[p. 78]{Kato}, we obtain
$\widetilde{T}^{(1)} = P K P$ and $\widetilde{T}^{(2)} = -PKSKP$ where $P = \begin{pmatrix} I_{n-k} & 0 \\ 0 & 0 \end{pmatrix}$ is the spectral projection corresponding to $\lambda = 0$ and $S= \begin{pmatrix} 0 & 0 \\ 0 & H_1^{-1} \end{pmatrix}$ is the Moore-Penrose pseudo-inverse of $H$.  Note that $T^{(1)} = PKP = 0$, and $T^{(2)} = -PKSKP = \begin{pmatrix} - K_0 H_1^{-1} K_0^* & 0 \\ 0 & 0 \end{pmatrix}$.  Since $\widetilde{T}^{(1)} = 0$, the expression for $\widetilde{T}^{(1)}(x)$ reduces to
$$
\widetilde{T}^{(1)}(x) = x(\widetilde{T}^{(2)} + \widetilde{T}^{(3)}x + \widetilde{T}^{(4)}x^2 + \ldots).
$$
Thus, the problem of finding a power series expression for the eigenvalues in the $\lambda$-group of $T(x)$ reduces to finding a power series expression for the eigenvalues of
$$
\widetilde{T}^{(2)}(x) = \widetilde{T}^{(2)} + \widetilde{T}^{(3)}x + \widetilde{T}^{(4)}x^2 + \ldots
$$
The second order coefficients for the expansions of the eigenvalues of $T(x)$ in the $\lambda$-group are the eigenvalues of $\tilde{T}^{(2)}$ which correspond to eigenvectors in the range of $P(0) = P$.  Those are the eigenvalues of the block $-K_0 H_1^{-1} K_0^*$.  Let $\lambda^{(2)}_j$ denote these eigenvalues.  Then the power series expansions for the eigenvalues of $T(x)$ in the $\lambda$-group have the form
$$\lambda(x) = \lambda + \lambda_j^{(2)}x^2 + O(x^3).$$
If the minimum eigenvalue of $-K_0 H_1^{-1} K_0^*$ is repeated, then we repeat the reduction process.  In this case, we need to find the spectral projection $\widetilde{P}^{(2)}$ corresponding to the minimal eigenvalue of $\widetilde{T}^{(2)}$.  

In the case when $n=4$, { we must have $k=2$ in order for the} block $-K_0 H_1^{-1} K_0^*$ { to have} a repeated eigenvalue. { Thus,} $\widetilde{P}^{(2)} = P = \diag(1,1,0,0)$.  For now, let us assume this to be the case.

Let $\lambda^{(2)}$ denote the repeated eigenvalue of the block $-K_0 H_1^{-1} K_0^*$. By the reduction process in \cite[II-\textsection 2.3]{Kato}, the eigenvalues of $\widetilde{T}^{(2)}(x)$ in the $\lambda^{(2)}$-group have the form
$$\lambda^{(2)} + x \lambda_j^{(3)} + O(x^2)$$
where $\lambda_j^{(3)}$ are the eigenvalues of the matrix
$$\widetilde{\widetilde{T}}^{(1)} = P \widetilde{T}^{(3)} P = \begin{pmatrix} K_0 H_1^{-1} K_1 H_1^{-1} K_0^* & 0 \\ 0 & 0  \end{pmatrix}$$
corresponding to the block $K_0 H_1^{-1} K_1 H_1^{-1} K_0^*$. In this case the $\lambda$-group eigenvalues have the form
$$\lambda(x) = \lambda + \lambda^{(2)} x^2 + \lambda_j^{(3)} x^3 + O(x^4).$$
If the third order coefficients of the two eigenvalues in the $\lambda$-group differ, then weak inverse continuity fails at $0$ by Theorem \ref{thm:characterization} and the remarks following the theorem.  If the two eigenvalues are the same, then the self-adjoint block $K_0 H_1^{-1} K_1 H_1^{-1} K_0^*$ is a multiple of the identity.  Since we have previously assumed that the eigenvalues of $K_0 H_1^{-1} K_0^*$ are repeated, that matrix must also be a multiple of the identity.  From this it follows that both $H_1$ and $K_1$ are scalar multiples of $K_0^*K_0$.  This implies that the matrix $A = H+iK$ is unitarily reducible, contradicting the hypothesis.  Therefore the condition of Theorem \ref{necandsuf} is both necessary and sufficient for weak continuity to fail when $n=4$.
\end{proof}

\section{Examples} 

\begin{example}\label{example12}
{ Following} Example 12 from \cite{CJKLS2}, { let us consider} 
\[A=\begin{pmatrix}
0&0&ik_1&0\\
0&0&0&ik_2\\
ik_1&0&1&ir\\
0&ik_2&ir&1\\
\end{pmatrix}
\]
for $k_1,k_2,r>0$, $k_1>k_2$. { Then} 
\[ { \re A} =\on{diag}(0,0,1,1)\text{ and } { \im A } =\begin{pmatrix}
0&0&k_1&0\\
0&0&0&k_2\\
k_1&0&0&r\\
0&k_2&r&0\\
\end{pmatrix}. \]
It is shown in \cite{CJKLS2} that $A$ is unitarily irreducible, and $0\in F(A)$ is the unique multiply generated boundary round point. Computing $K_0H_1^{-1}K_0^*=\on{diag}[k_1^2,k_2^2]$ gives a simple largest eigenvalue, so Theorem \ref{necandsuf} shows that weak continuity holds at the origin whereas strong continuity fails by Theorem \ref{thm:multiGen}.

{ The failure of strong continuity in this setting was established in \cite{CJKLS2} via a case specific parametrization of $\partial F(A)$; the persistence of weak continuity was also mentioned there in passing.} 
\begin{figure}[h]
\includegraphics[scale=.55]{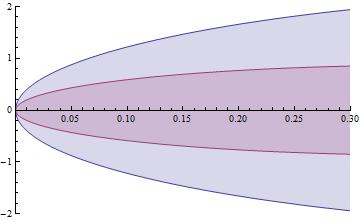}
\caption{$F(A)$ when $k_1=2$ and $k_2=1=r$. The curves $z_k(\theta)$ are illustrated for the two largest eigenvalues of $\on{Re}(e^{-i\theta}A)$, which coincide at $\theta=\pi$ but split into two distinct branches away from $\pi$.}
\end{figure}
\end{example}

\begin{example} \label{ex:weakfail} Weak continuity can fail for a unitarily irreducible 4-by-4 matrix. 

Let $H=\on{diag}(0,0,1,4)$ and $K=\begin{pmatrix}
0&0&1&0\\
0&0&0&2\\
1&0&1&2\\
0&2&2&3\\
\end{pmatrix}$. Obviously, $e_3$ is a simple eigenvector for $H$, and the set $\{e_3,Ke_3,K^2e_3,K^3e_3\}$ is linearly independent, establishing the unitary irreducibility of $A =H+iK$. Also, $H_1=K_0^*K_0$, so by Theorem \ref{necandsuf}, weak continuity fails. See Figure \ref{fig:critCurves} below.

\begin{figure}[h]
\begin{tabular}{cc}
\includegraphics[scale=0.3]{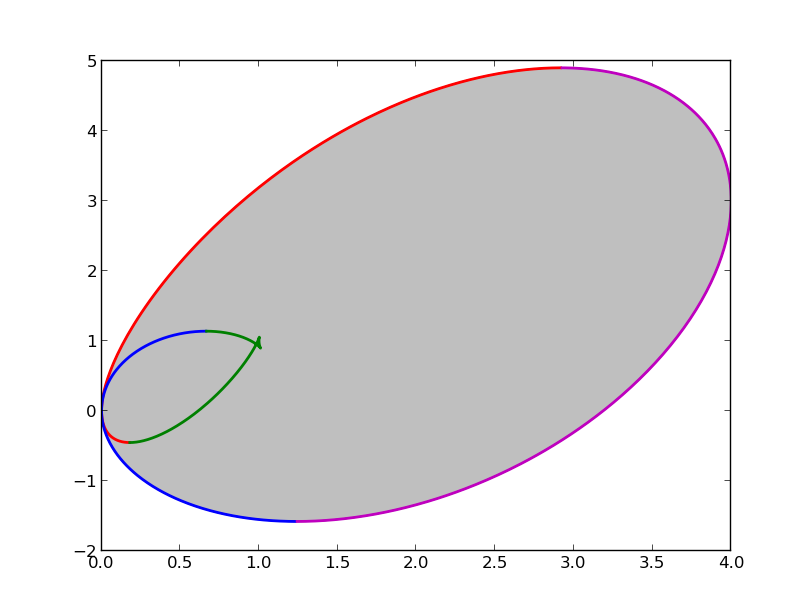} & \includegraphics[scale=0.3]{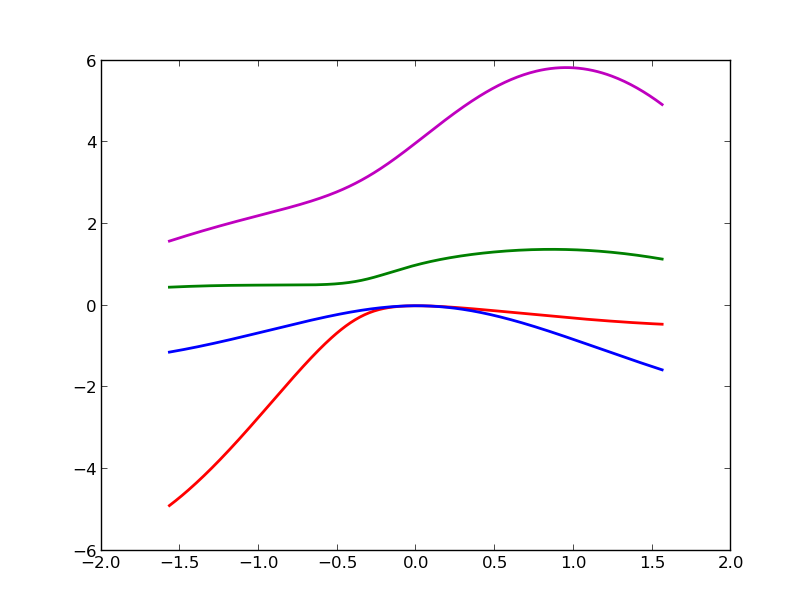}
\end{tabular}
\caption[]{The numerical range of $A$ from Example \ref{ex:weakfail} including the critical curves $z_k(\theta)$ (left) and the corresponding critical values $\lambda_k(\theta)$ (right) for $\theta \in [-\pi/2,\pi/2)$.
} \label{fig:critCurves}
\end{figure}
\end{example}

%\vfill

\begin{example}\label{6x6}
Let us now revisit \cite[Example 6.1]{LLS12}: 
\eq{sixeq}
 A=\begin{pmatrix}
 0&x&0&cy&0&0\\
 0&0&y&0&0&0\\
 0&0&0&0&0&0\\
 0&0&-cx&0&\sqrt{1-c^2}\xi&0\\
 0&0&0&0&0&\eta\\
 0&0&0&0&0&0\\
 \end{pmatrix}, \en
where $w,y,\xi,\eta,c>0, \;w^2+y^2=\xi^2+\eta^2=4,\;c<1$.

According to \cite{LSS},  $F(A)$ is the unit disk and $A$ is unitarily irreducible. Also, the eigenvalues of $\re (e^{-i\theta}A)$ do not depend on $\theta$ and equal $\pm 1$ (each having multiplicity 2) and $\pm c\eta/2$. So, condition (1) of Theorem~\ref{thm:characterization} holds. Thus, $f_A^{-1}$ is strongly continuous everywhere on $F(A)$ in spite of the fact that each point of $\partial F(A)$ is round and multiply generated.
\end{example}

\begin{example} \label{ex:weakfail2}
The following 5-by-5 example illustrates that weak continuity can fail even when the eigenfunctions corresponding to a fully round boundary point split at an even degree. It is based on Example \ref{ex:weakfail} with the addition of a normal eigenvalue at 0.  Let $A, H$, and $K$ be as in Example \ref{ex:weakfail}. The eigenvalue $\lambda = 0$ of $A(t) = H+tK$ has multiplicity 2 at $t=0$. Following the proof of Theorem \ref{necandsuf}, the $\lambda$-group eigenvalues of $A(t)$ have the following Taylor series expansions:
$$ \lambda_1(t) = -t^2 + (\tfrac{7+\sqrt{65}}{8})t^3 + O(t^4),$$
$$ \lambda_2(t) = -t^2 + (\tfrac{7-\sqrt{65}}{8})t^3 + O(t^4).$$
These two eigenfunctions split at degree 3.  
Now, consider the 5-by-5 matrix $B = A \oplus (0)$. The numerical range of $B$ is the same as that of $A$. The pencil $B(t) = (H\oplus(0)) + t (K\oplus (0))$ adds an additional lambda group eigenvalue:
$$ \lambda_3(t) = 0.$$
This $\lambda$-group splits at degree 2, but weak continuity still fails at the point $z = 0$ in the boundary of the numerical range of $B$ since no one eigenfunction is minimal in a neighborhood of $t = 0$.  

\begin{figure}[h]
\begin{center}
\begin{tabular}{cc}
\includegraphics[scale=0.3]{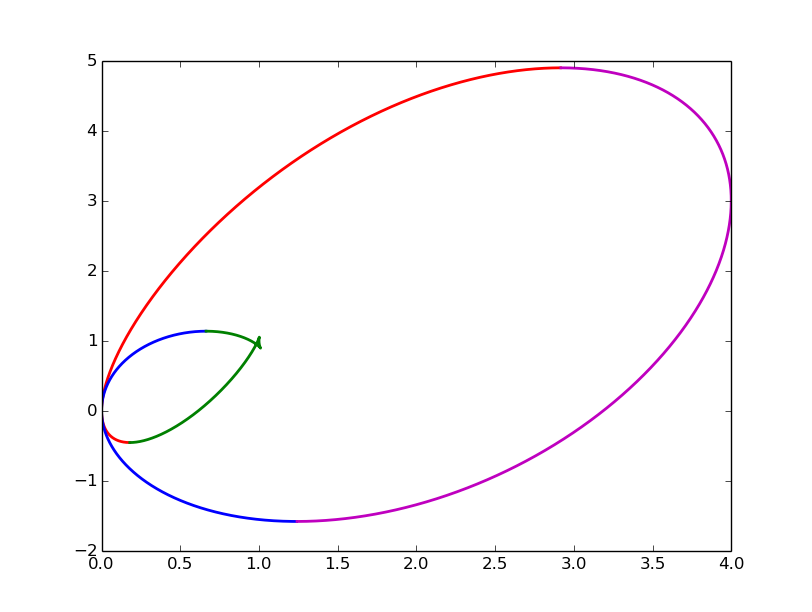} & \includegraphics[scale=0.3]{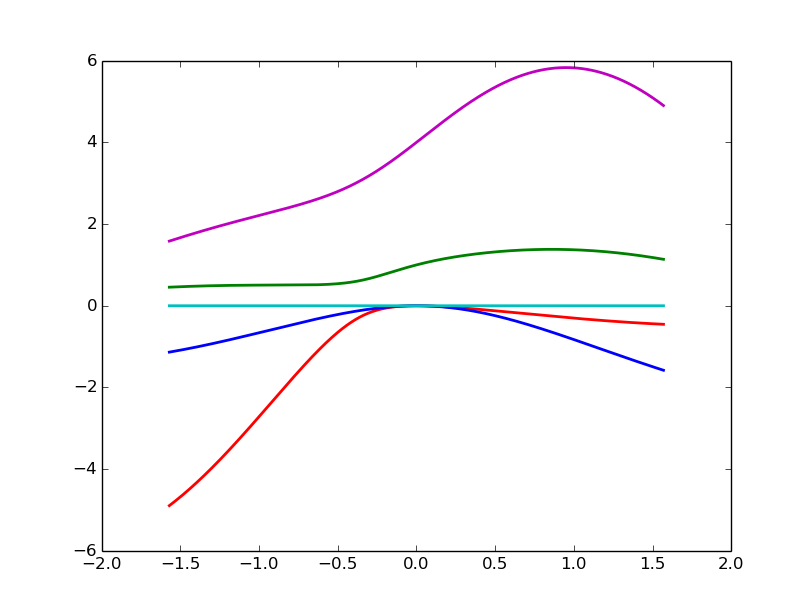}
\end{tabular}
\caption[]{The critical curves $z_k(\theta)$ (left) and critical values $\lambda_k(\theta)$ (right) with $\theta \in [-\pi/2,\pi/2)$ for the matrix $B$ in Example \ref{ex:weakfail2}. Note that the $\lambda = 0$ group eigenvalues split at degree 2 when $\theta=0$, but weak continuity of $f_B^{-1}$ fails at the corresponding point $z=0$ on the boundary of $F(B)$ since no one eigenfunction is minimal in a neighborhood of $\theta=0$.  
} \label{fig:critCurves2}

\end{center}
\end{figure}
\end{example}

In order for weak continuity to fail at a point $z$ where the corresponding eigenvalue functions split at an even degree, as in the example above, there must be at least three distinct eigenvalue functions that correspond to $z$.  If this occurs for a 4-by-4 matrix, the condition of Theorem \ref{necandsuf} applies and we see that weak continuity must hold. Therefore the original version of Theorem \ref{thm:characterization}(2) is valid when $n \le 4$.  

\begin{example} \label{ex:irred}
The following 6-by-6 matrix is unitarily irreducible and the eigenfunctions corresponding to $z=0$ split at degree two, but $f_A^{-1}$ is not weakly continuous there.   
$$A=\begin{pmatrix}
0 & 0 & 0 & i & 0 & 0 \\
0 & 0 & 0 & 0 & 2i & 0 \\
0 & 0 & 0 & 0 & 0 & 2i \\
i & 0 & 0 & 1 + 1.3i & 1.1i & 0.3i \\
0 & 2i & 0 & 1.1i & 1 + 1.2i & 1.7i \\
0 & 0 & 2i & 0.3i & 1.7i & 1 + 1.7i
\end{pmatrix}$$
If $H = \re(A)$ and $K = \im(A)$, then the Taylor series expansions for the 0-group eigenvalues of $H+tK$ about $t=0$ are:
\begin{align*}
\lambda_1(t)&= -4t^2 - 1.6689t^3 + O(t^4), \\
\lambda_2(t)&= -4t^2 + 1.5862t^3 + O(t^4), \text{ and}  \\
\lambda_3(t)&= -t^2 + 12.9826t^3 + O(t^4). \\
\end{align*}
\begin{figure}[h]
\begin{center}
\begin{tabular}{cc}
\includegraphics[scale=0.3]{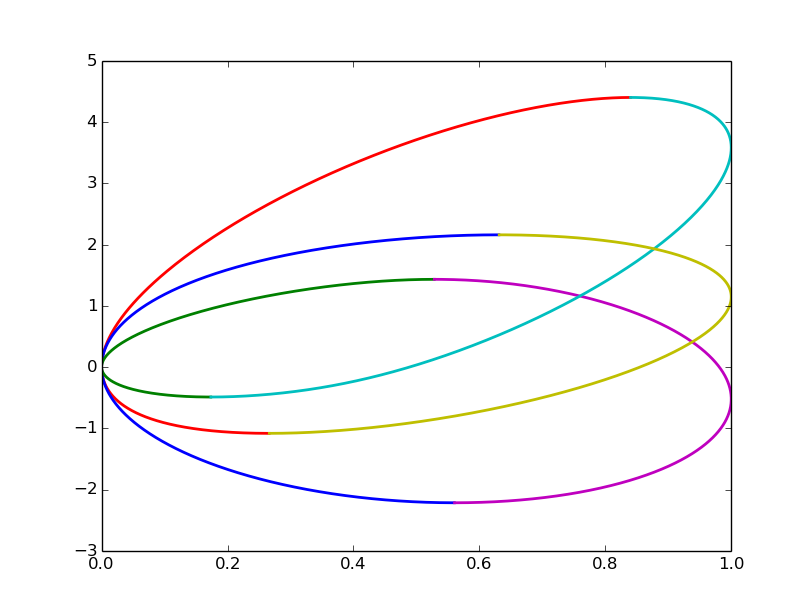} & \includegraphics[scale=0.3]{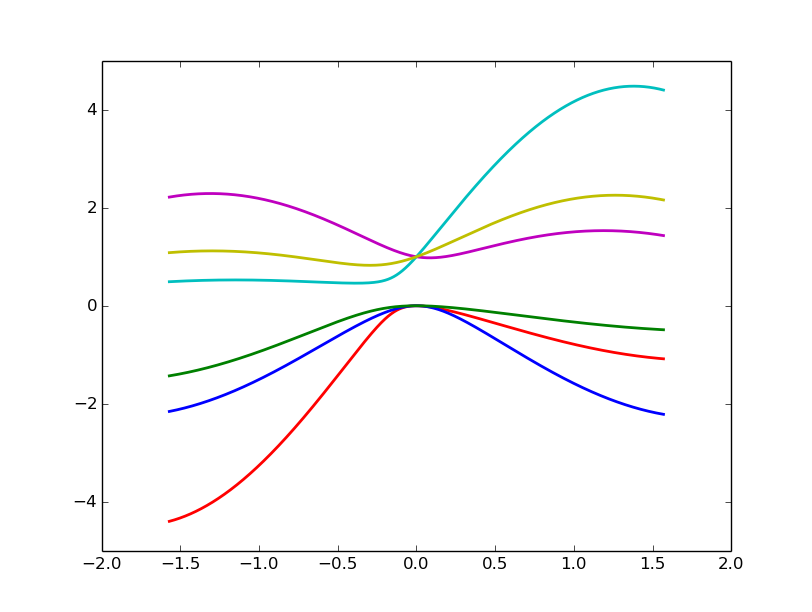}
\end{tabular}
\caption[]{The critical curves and values of the matrix $A$ in Example \ref{ex:irred}.} \label{fig:irred}
\end{center}
\end{figure}
\end{example}

\begin{ack}
We are very grateful to Stephan Weis for pointing out the deficiency in the original
statement of Theorem 2.1 and for follow up discussions since then.  
\end{ack}
\color{black}
\newpage

\bibliography{master}
\bibliographystyle{plain}

\end{document}